\RequirePackage{fix-cm}
\documentclass[smallextended]{svjour3}       
\smartqed  
\usepackage{graphicx}
\usepackage{lineno,hyperref}
\usepackage{amsmath,amssymb}
\usepackage{enumerate}
\usepackage{multicol}
\usepackage{graphicx}
\usepackage{epsfig}
\usepackage{subfigure}
\usepackage{ntheorem}
\usepackage{color}
\numberwithin{equation}{section}
\numberwithin{table}{section}

\newtheorem{thm}{Theorem}[section] 
\newtheorem{defn}[thm]{Definition} 
\newtheorem{lem}[thm]{Lemma} 
\newtheorem{alg}[thm]{Algorithm}
\numberwithin{thm}{section}

\def\C{{\bf C}}

\newcommand{\be}{\begin{equation}}
\newcommand{\ee}{\end{equation}}

\begin{document}

\title{On   condition numbers of the total least squares problem with linear equality constraint\thanks{
This paper is supported in part by the National Natural Science Foundation of China under grants 12090011,
11771188; the Priority Academic Program Development Project (PAPD); the Top-notch Academic
Programs Project (No. PPZY2015A013) of Jiangsu Higher Education Institution.}
}


\author{Qiaohua Liu        \and
        Zhigang Jia
}


\institute{Qiaohua Liu \at
              Department of Mathematics, Shanghai University, Shanghai 200444, P.R. China\\
             \email{qhliu@shu.edu.cn}           
           \and
           Zhigang Jia \at
        School of Mathematics and Statistics, Jiangsu Normal University,
Xuzhou 221116, P. R. China\\
Research Institute of Mathematical Science, Jiangsu Normal University,
Xuzhou 221116, P. R. China\\
\email{zhgjia@jsnu.edu.cn}
}

\date{Received: date / Accepted: date}

\maketitle

\begin{abstract}
This paper is
devoted to   condition numbers  of the  total least squares  problem with linear equality constraint (TLSE).
With novel limit techniques, closed formulae for   normwise, mixed and componentwise condition numbers of the TLSE problem are derived.
Computable expressions and upper bounds for these condition numbers are also given to avoid the costly Kronecker product-based operations.
The results unify the ones for the TLS problem.
For TLSE problems with equilibratory input data,  numerical experiments  illustrate that normwise condition number-based estimate is sharp to
evaluate the forward error of the solution, while for sparse and badly scaled matrices,
mixed and componentwise condition numbers-based estimates  are much tighter.


\end{abstract}

\section{Introduction}

In  many data fitting and estimation problems, total least squares model (TLS) \cite{pe} is used to find a ``best'' fit to the  overdetermined system $Ax\approx b$, where $A\in \mathbb{R}^{q\times n}(q>n)$   and $b\in \mathbb{R}^q$   are contaminated by some noise.
The model determines perturbations $E\in \mathbb{R}^{q\times n}$ to the coefficient matrix $A$ and $f\in \mathbb{R}^q$ to the vector $b$ such that
\begin{equation}
\min_{E,f}\|[E\quad f]\|_{F}, \qquad \mbox{subject\ to}\quad (A+E)x=b+f,\label{1.2}
\end{equation}
where $\|\cdot\|_F$ denotes the Frobenius norm of a matrix. The TLS model was carefully proposed in 1901 \cite{pe}. However it was  not extensively explored for a long time.
In recent forty years,  it has been widely applied in a broad class of scientific disciplines such
as system  identification \cite{ld}, image  processing \cite{nbk,npp},  speech and audio processing \cite{hv,lmv}, etc.
A wide range of concern on TLS owes to Golub and Van Loan, who introduced    the TLS model to the numerical linear algebra area in 1980 \cite{gv}, and they developed an algorithm for solving the TLS problem through  the singular value decomposition (SVD) of $[A~ b]$.   When $A$ is large,   a complete SVD will be very costly.  One improvement is to compute a partial SVD
based on Householder transformation \cite{va}   or Lanczos bidiagonalization introduced by Golub and Kahan  \cite{gv2}.  Another improvement is based on the
Rayleigh quotient iteration \cite{bj}, or Gauss-Newton iteration \cite{ff,slc}. Recently,
a randomized algorithm \cite{xxw2} presented by Xie el al. in 2018    greatly
reduces the computational time and still maintains good accuracy with very high probability.   For a comprehensive reading about the TLS model, we refer to \cite{mv,val,vh1}.

In 1992,   Dowling et al. \cite{ddl} studied the TLS model with  linear equality constraint (TLSE):
\begin{equation}
\min_{E,f} \|[E\quad f]\|_{F}, \quad \mbox{subject\ to}\quad (A+E)x=b+f,\quad Cx=d,\label{1.6}
\end{equation}
where $C\in {\mathbb R}^{p\times n}$ is of full row rank and $\Big[{C\atop A}\Big]$  has   full column rank.
They proposed to solve it on the  basis of QR and SVD matrix factorizations.
Further investigations on problem TLSE  were performed in  \cite{sc}, where   iteration methods  were derived based on the Euler-Lagrange theorem.
 Recently, Liu et al. \cite{liu} interpreted the TLSE solution as an approximation of the solution to an unconstrained weighted TLS  (WTLS) problem, with a large weight
 assigned on the constraint, based on which a QR-based inverse iteration (QR-INV) method was presented.

For the sensitive analysis of a problem,   the condition number measures
the worst-case sensitivity of its solution to small perturbations in the input data. Combined with backward error
estimate,   an approximate upper bound can be derived for the forward error, that is, the difference between a perturbed
solution and the exact solution.  There are a lot of work studying condition numbers of the standard TLS problem, see \cite{bg,jl,lj,mzw,wei0,wei,xxw,zlw,zmw}.
As far as we know,    condition numbers   of the TLSE problem
have not been studied in  the literature. As a continuation of the previous
work \cite{liu}, in this paper we will investigate this issue.

In \cite{sc}, by applying the method of Lagrange multipliers, Schaffrin   proved that the TLSE solution satisfies the following generalized eigenvalue problem
$$
\left[\begin{array}{ccc}
A^TA\quad&A^Tb\quad&C^T\\b^TA\quad&b^Tb\quad&d^T\\C\quad&d\quad&0\end{array}\right]\left[\begin{array}c x\\-1\\ \lambda\end{array}\right]=\nu_{\min}^2\left[\begin{array}{ccc} I_n\quad&&\\&1&\\&&0_p\end{array}\right]\left[\begin{array}c x\\-1\\\lambda\end{array}\right],
$$
where $\lambda$ is a quantity related to the vector of Lagrange multipliers, and
$\nu_{\min}^2$ is the smallest generalized eigenvalue. Notice that the associated matrix ${\rm diag}(I_{n+1}, 0_p)$  in the generalized eigenvalue problem  is positive semidefinite, which causes difficulties in the perturbation analysis of the TLSE problem. To the best of our knowledge, there are no good results about perturbation theory of generalized eigenvalue problem $Mx=\mu Nx$ with $N$ being positive semidefinite.

In \cite{ri}, Rice  gave a general theory of condition numbers. If $x=\psi(a)$ is continuous and Fr{\rm \'{ e}}chet differentiable function
mapping from ${\mathbb R}^s$ to ${\mathbb R}^t$, where $a$ is a parameter related to input data. For small perturbations $\delta a$,  denote $\delta x=\psi(a+\delta a)-\psi(a)$, then according to \cite{ri}, the relative normwise condition number of $\psi$ at $a$ is
$$ \kappa^{\rm rel}(\psi,a):=\lim_{\varepsilon \rightarrow 0}\sup_{\|\,\delta a\,\|_2\leq\varepsilon\|a\|_2}
\frac{\|\delta x\|_2/\|x\|_2}{\|\delta a\|_2/\|a\|_2}=\frac{\|\psi\prime(a)\|_2\|a\|_2}
{\|\psi(a)\|_2},
$$
where $\psi(a)\not =0$,  $\psi\prime(a)$ denotes the Fr$\acute{\rm e}$chet derivative \cite{ri} of $\psi$ at the point $a$.
In \cite{bg,jl,lj,mzw,xxw,zlw,zmw},
condition number of TLS problems  are generally based on Rice's theory.

 Although a closed form TLSE solution with Moore-Penrose inverse operation was given in \cite{liu}(see (\ref{2.4})),
it is not easy to get   a  simple  $\psi$ to derive its Fr\'{e}chet derivative   of the TLSE solution.
Fortunately,  the TLSE solution can be approximated by an unconstrained WTLS solution in a limit sense \cite{liu}, while
 the perturbation of the standard TLS problem is widely investigated.
  In the literature, there are many similar problems whose  perturbation analysis is derived based on a limit technique, say for
  equality constrained least squares problem  by Wei and De Pierro \cite{pw,wei2,wp}, Liu and Wei \cite{liuw}; for
   mixed least squares-total least squares problem   by Zheng and Yang \cite{zy}.
In this paper, we consider modifying the perturbation results in \cite{bg,jl,lj}
   to be accessible for our problem,  and then using
  the limit technique to derive the first order perturbation result of the TLSE solution, from which closed formulae
  of normwise, mixed and componentwise condition numbers are derived.

 The organization of this paper is as follows. In Section 2, we present some  preliminary results about TLS and TLSE problems. The first order perturbation result for the TLSE problem
 is investigated in Section 3. Moreover, the Kronecker-product-based normwise, mixed and componentwise condition number formulae for the TLSE problem are given.
 To make the formulae more computable, Kronecker-product-free upper bounds for the normwise, mixed, and componentwise condition numbers are presented in Section 4. The perturbation estimates and condition number formulae for the standard TLS problem can be recovered from our results.
In Section 5, some numerical examples are provided to  demonstrate that our upper bounds are tight. Some concluding remarks are given in Section 6.

 Throughout the paper,  $\|\cdot\|_2$ denotes the Euclidean vector or matrix  norm,  $I_n$, $0_n$, $0_{m\times n}$ denote the $n\times n$ identity matrix,  $n\times n$ zero matrix,
and $m\times n$ zero matrix, respectively. If subscripts are ignored, the sizes of identity and zero matrices are consistent with the context.
   For a matrix $M\in {\mathbb R}^{m\times n}$,  $M^T$, $M^\dag$, $\sigma_j(M)$ denote the transpose, the  Moore-Penrose inverse,
    the $j$-th largest singular value of $M$, respectively. {\rm vec}($M$) is an operator, which stacks the columns of $M$ one  underneath the other. The Kronecker product   of $A$ and $B$ is defined by $A\otimes B=[a_{ij}B]$ and its property is listed as follows \cite{gr,ls}:
$$
\begin{array}l
{\rm vec}(AXB)=(B^{T}\otimes A){\rm vec}(X),\qquad (A\otimes B)^{T}=A^{T}\otimes B^{T}.
\end{array}
$$

\section{Preliminaries}

In this section we first recall some well known results about TLS and TLSE problems \cite{ddl,gv}. In order to apply the TLS theory to analyze the perturbation of  problem  TLSE conveniently, we use $[L~ h]$ as the input data and write the TLS model as the following form:
\begin{equation}
\min\|[E\quad f]\|_F,\qquad {s.t.}\qquad (L+E)x=h+f,\label{1.0}
\end{equation}
where $L\in {\mathbb R}^{m\times n}$, $h\in {\mathbb R}^{m}$($m> n$).
Following \cite{gv}, if the SVD \cite[Chapter 2.4]{gv2} of $[L\quad h]$ is given by
\begin{equation}
[L\quad h]=U\Sigma V^T=\sum\limits_{i=1}^{n+1}\sigma_iu_iv_i^T,\quad \Sigma=\mbox{diag}(\sigma_1,\sigma_2,\cdots,\sigma_{n+1}),\label{1.01}
\end{equation}
where $\sigma_i=\sigma_{i}([L\quad h])$ and $\sigma_1\geq \sigma_2\geq \cdots\geq \sigma_{n+1}>0,$  then under the Golub-Van Loan's genericity condition \cite{gv}:
\begin{equation}\sigma_n(L)>\sigma_{n+1}([L\quad h])=\sigma_{n+1},\label{1.3}\end{equation}
the right singular vector $v_{n+1}$   contains  a nonzero last component, from which the TLS solution is uniquely determined by normalizing the last component to $-1$. Moreover  the TLS solution satisfies the following augmented system
\begin{equation}
\left[\begin{array}{cc}L^TL&~~L^Th
\\ h^TL&~~h^Th
\end{array}\right]\left[\begin{array}c
x\\ -1
\end{array}\right]=\sigma_{n+1}^2
\left[\begin{array}c
x\\ -1
\end{array}\right],\label{1.50}
\end{equation}
from which
\begin{equation}
L^Tr=\sigma_{n+1}^2x,\qquad for \qquad r=Lx-h,\label{1.5}
\end{equation}
and the closed form of the TLS solution can  be expressed by
\begin{equation}
x_{\rm tls}=(L^TL-\sigma_{n+1}^2I_n)^{-1}L^Th.\label{tls}
\end{equation}
Let $[L\quad h]$ be perturbed to   $[L+\Delta L\quad h+\Delta h]$, where  $\|[\Delta L\quad \Delta h]\|_F$ is sufficiently small such that
the genericity condition still holds for the perturbed TLS problem,  then for the unique solution  $\widehat x_{\rm tls}$
of the perturbed TLS problem,  Zhou et al.  \cite{zlw} first made the first order perturbation estimate based on (\ref{tls}) and gave the explicit expressions of condition numbers.
Li and Jia \cite{lj} presented a new and simple closed form formula with different approach and proved that $\Delta x=\widehat x_{\rm tls}-x_{\rm tls}$ satisfies
\begin{equation}
\Delta x=K{\rm vec}([\Delta L\quad \Delta h])+{\cal O}(\|[\Delta L\quad \Delta h]\|_F^2),\label{2.100}
\end{equation}
where  $x=x_{\rm tls}$ is the exact TLS solution, and with $r=Lx-h$, $P=L^{T}L- {\sigma}^{2}_{n+1}I_n$,
\begin{equation}
K=K_{\rm LJ}=P^{-1}\left(\left(\frac{2L^Trr^{T}}{\|r\|_2^2}-L^{T}\right)([x^{T}\quad -1]\otimes I_m)-[I_n \quad 0_{n\times 1}]\otimes r^{T}\right).\label{3.1.1}
\end{equation}
They also proved that this closed formula is equivalent to the one from Zhou et al. \cite{zlw}. In the same year, Baboulin and Gratton derived another formula  of $K$  in \cite{bg}:
$$
K_{\rm BG}=\big[-(x^{T}\otimes D)-(r^T\otimes P^{-1})\Pi_{m,n} \qquad D\big],
$$
where     $ D=P^{-1}(L^T-2{xr^T\over \rho^2} ),$  $\rho=(1+\|x\|_2^2)^{1\over 2}$,
and  $\Pi_{m,n}$ is a vec-permutation matrix such that ${\rm vec}(Z^T)=\Pi_{m,n}{\rm vec}(Z)$  for an arbitrary $m\times n$ matrix $Z$. In \cite{xxw}, Xie et al. proved that
$\|K_{\rm BG}\|_2=\|K_{\rm LJ}\|_2$ and the associated normwise condition numbers are equivalent. In \cite{liuc}, Liu et al.  proved that
\begin{equation}
K_{\rm BG}=\big[-(x^{T}\otimes D)-P^{-1}(I_n\otimes r^T ) \quad D\big]=K_{\rm LJ}.\label{2.001}
\end{equation}

For the TLSE problem (\ref{1.6}), denote $\widetilde A=[A ~ b], \widetilde C=[C\quad d]$.  In \cite{ddl},  the QR-SVD procedure for computing the TLSE solution first factorizes  $\widetilde C^T$ into the QR form:
\begin{equation}\widetilde C^T=\widetilde Q \widetilde R=[\widetilde Q_1\quad \widetilde Q_2]\left[{\widetilde R_1\atop 0}\right]=\widetilde Q_1\widetilde R_1,\quad\mbox{with}~ \widetilde Q_1\in{\mathbb R}^{ (n+1)\times p}, \widetilde R_1\in{\mathbb R}^{ p\times p},\label{1.60}
 \end{equation}
 and then computes the SVD of $\widetilde A\widetilde Q_2$ as
 \begin{equation}
 \widetilde A\widetilde Q_2=\widetilde U\widetilde\Sigma\widetilde V^T=\sum\limits_{i=1}^{n-p+1}\widetilde\sigma_i\widetilde u_i\widetilde v_i^T.\label{1.61}
\end{equation}
If $\widetilde Q_2\widetilde v_{n-p+1}$ contains a nonzero last component, a  TLSE solution $x=x_{\rm tlse}$ is   determined by normalizing the last component in $\widetilde Q_2\widetilde v_{n-p+1}$  to $-1$, i.e.,
 \begin{equation}
 \Big[{x\atop -1}\Big]=\rho \widetilde Q_2\widetilde v_{n-p+1},\label{2.2}
 \end{equation}
 where $\rho=(1+\|x\|_2^2)^{1/2}$ up to a factor $\pm 1$.

In \cite{liu}, Liu et al. carried out further investigations on the uniqueness condition and the explicit  closed form   for the TLSE solution.
  For the thin QR factorization of $C^T$:
  \begin{equation}
C^T=QR=[Q_1\quad Q_2]\left[{R_1\atop 0}\right]=Q_1R_1, \label{2.00}
\end{equation}
let $x_{\rm C}=C^\dag d=Q_1R_1^{-T}d$ be the minimum 2-norm solution to $Cx=d$ and set $r_{\rm C}=Ax_{\rm C}-b$. Note that  $Q_2$ and the specific matrix
\begin{equation}
\widetilde Q_2=\left[\begin{array}{cc}Q_2\quad&\zeta x_{\rm C}\\ 0\quad&-\zeta\end{array}\right],~\mbox{with}~~ \zeta=\Big(1+\|x_{\rm C}\|_2^2\Big)^{-1/2},\label{2.01}
\end{equation}
have orthonormal columns  and the spans of the columns are  the null space of $C$ and $\widetilde C$, respectively. Then under the  genericity condition
\begin{equation}
\sigma_{n-p}(AQ_2)=:\overline{\sigma}_{n-p}>\widetilde \sigma_{n-p+1}:=\sigma_{n-p+1}([AQ_2\quad \zeta r_{\rm C}]),\label{2.1}
\end{equation}
it was proved that there must be a  nonzero last component in $\widetilde Q_2\widetilde v_{n-p+1}$, and hence  the TLSE problem has a unique solution taking the form
\begin{equation}
x_{\rm tlse}=C^\dag d-Q_2S_{11}^{-1}Q_2^TA^Tr_{\rm C}=C_A^\dag d+{\cal K}A^Tb,\label{2.4}
\end{equation}
where $S_{11}= Q_2^TA^TAQ_2-\widetilde\sigma_{n-p+1}^2I_{n-p}$, and ${\cal K}=Q_2S_{11}^{-1}Q_2^T$. Let ${\cal P}=I_n-C^\dag C$,  then it is obvious that
$$
C_A^\dag=(I_n-{\cal K}A^TA)C^\dag,\qquad {\cal K}=({\cal P}(A^TA-\widetilde\sigma_{n-p+1}^2I_n){\cal P})^\dag.
$$
With this,   the TLSE solution can be regarded as the
limit case of the solution to an unconstrained weighted TLS (WTLS) problem \cite{liu}:
\begin{equation}
\min_{\hat E, \hat f}\|[\hat E\quad \hat f]\|_F \qquad \mbox{subject to}\qquad (L_\epsilon+\hat E)x_\epsilon=h_\epsilon+\hat f,\label{2.11}
\end{equation}
as $\epsilon$ tends to zero, where
\begin{equation}
L_\epsilon=W_{\epsilon}^{-1}L=\left[\begin{array}c \epsilon^{-1}C\\ A\end{array}\right],\qquad
h_\epsilon=W_{\epsilon}^{-1}h=\left[\begin{array}c \epsilon^{-1}d\\ b\end{array}\right],\quad \mbox{for}\quad W_{\epsilon}=\left[\begin{array}{cc} \epsilon I_p\quad&0\\0\quad& I_{q}\end{array}\right].\label{2.12}
\end{equation}
Under the genericity assumption (\ref{2.1}) and the assumption
\begin{equation}
0<2\epsilon^2\|\widetilde C^\dag\|_2^2\|\widetilde A\|_2^2<\overline \sigma_{n-p}^2-\widetilde\sigma_{n-p+1}^2, \label{2.14}
\end{equation}
the WTLS solution is uniquely determined by $x_\epsilon=(L_\epsilon^TL_\epsilon-\widetilde\sigma_\epsilon^2 I_n)^{-1}L_\epsilon^Th_\epsilon$, where
$\widetilde \sigma_\epsilon$ is the smallest singular value of $[L_\epsilon\quad h_\epsilon]$ and
\begin{equation}
\lim_{\epsilon\rightarrow 0+}\widetilde \sigma_\epsilon=\widetilde\sigma_{n-p+1},\qquad \lim_{\epsilon\rightarrow 0+}x_\epsilon=x_{\rm tlse}.\label{lim}
\end{equation}

In \cite{liu}, Liu et al. presented the QR-based inverse iteration (QR-INV) for above weighting problem to get the TLSE solution, in which the initial QR factorization of $[L_\epsilon\quad h_\epsilon]$ and the solution of two $(n+1)\times (n+1)$ triangular systems in each iteration loop are needed. This is costly  when $n$ increases.
A recent  study of the randomized TLS with
Nystr$\ddot{\rm o}$m scheme (NTLS)   proposed by Xie et al. \cite{xxw2} is also appropriate for solving above weighted TLS problem, and it is more adaptable than QR-INV in getting the minimum-norm solution, when $\widetilde A\widetilde Q_2$ has multiple smallest singular values, and the TLSE solution is not unique.

In the NTLS algorithm,  the known or estimated rank is set as the regularization parameter, and  the QR factorization and SVD are implemented on  much smaller matrices. Most flops are spent on the matrix-matrix multiplications, which are the so-called BLAS-3 operations,  resulting in potential efficiency of the algorithm. The algorithm is described as follows.\\

\begin{alg}[\cite{xxw2}]\label{alg2.1} Randomized algorithm for WTLS problem (\ref{2.11}) via Nystr$\ddot{\rm o}$m scheme (NWTLS).

Inputs: $C\in {\mathbb R}^{p\times n}, d\in {\mathbb R}^{p}, A\in {\mathbb R}^{q\times n}, b\in {\mathbb R}^{q}$, weighting factor $\epsilon$, and $k<\ell$($\ell\ll n$),
where $\ell$ is the sample size.

Output:  an approximated solution $x_{\rm nwtls}$ for the TLSE problem.

1. Solve  $(G^TG)X=\Omega$ where $G=[L_\epsilon\quad h_\epsilon]$ is defined in (\ref{2.12}), and $\Omega$ is an $(n+1)\times \ell$ Gaussian matrix generated via Matlab command {\sf randn}($n+1,\ell$).

2. Compute the $(n+1)\times \ell$ orthonormal matrix $Q$ via QR factorization $X=QR$.

3. Solve $(G^TG)Y=Q$, and form the $\ell\times \ell$ matrix $Z=Q^TY$.

4. Perform the Cholesky factorization $Z=J^TJ$, and seek $K$ by solving $KJ=Y$.

5. Compute the SVD: $K=V\Sigma U^T$, and  form the solution $x_{\rm nwtls}=-v(1:n)/v(n+1)$, where $v=V(:,1)$.
\end{alg}

\begin{remark} Algorithm \ref{alg2.1} is refined from the conventional randomized SVD (RSVD) of $(G^TG)^{-1}$. The RSVD   \cite{hmt} of a matrix $M\in {\mathbb R}^{m\times (n+1)}$  usually
starts from computing
 the orthonormal column basis of the projected matrix $M\Omega$ with  $\Omega$ is an $(n+1)\times \ell$ Gaussian matrix.  The approximate SVD of $  M\approx QQ^TM$ is then obtained by computing the SVD of a smaller matrix $Q^TM$. The RSVD algorithm approximates well  large  singular values and   corresponding singular
vectors \cite[Thms.10.5,10.6]{hmt}, especially for the matrix with fast decay rate in its  singular values.

In Algorithm  \ref{alg2.1}, for the sake of stability, the system $(G^TG)X=H$ could be solved by computing  the QR factorization of $G$ and then solving two $\ell\times \ell$ triangular systems, where the parameter $\ell=k+t$ stands for the number of sampling,   and  $k$ is the number of singular values and singular vectors we expect to compute.
According to the proof of \cite[Theorem 3.1]{liu}, we know that $G=[L_{\epsilon}\quad  h_\epsilon]$  has $p$ singular values of order ${\cal O}(\epsilon^{-1})$ and $n-p+1$ singular values approximating
those of $\widetilde A\widetilde Q_2$, therefore the matrix $(G^TG)^{-1}$ has at least $n-p+1$  dominant singular values much larger than those $p$   singular values of ${\cal O}(\epsilon^2)$.  In order
 for  a higher accuracy of the algorithm, we can choose $\ell=k+t$ with $k=n-p+1$, and the oversampling factor $t$, say $t=5$.
Moreover, if we can predict  the number  $s$  of   dominant   singular values  for $\widetilde A\widetilde Q_2$ that are much smaller than $n-p+1$,
 then we can choose the parameter  $k=s$  to reduce the computational cost.
\end{remark}

\section{Condition numbers of the TLSE problem}

Condition numbers  measure the sensitivity of the solution to the original data in problems, and they play an
important role in numerical analysis.  For the TLSE problem,  let $m=p+q$ and  $L, h$ be  defined in (\ref{2.12}), define the mapping  $\phi: {\mathbb R}^{m\times n}\times {\mathbb R}^m\rightarrow {\mathbb R}^n$ by
\begin{equation}
\phi([L\quad h])=x=x_{\rm tlse}.\label{3.1}
\end{equation}
According to Rice's theory of condition numbers \cite{ri},  the Fr\'{e}chet  derivative  $\phi'([L\quad h])$ such that
\begin{eqnarray}
\Delta x&=&\phi([L+\Delta L\quad h+\Delta h])-\phi([L\quad  h])\nonumber\\
&=&\phi'([L\quad  h])\cdot [\Delta L\quad \Delta h]+{\cal O}(\|[\Delta L\quad\Delta h]\|_F^2)\nonumber
\end{eqnarray}
is necessary. As pointed out previously,  there are some difficulties in constructing  a simple expression  $x=\phi([L\quad h])$  and computing  $\phi'([L\quad h])$ directly.
Instead, we  can start from the differentiability
of the weighted TLS solution $x_{\epsilon}$ and then take the limits as the parameter $\epsilon$ approaches
zero to get the various condition numbers of the TLSE solution.

To this end,  let $L, h, L_\epsilon, h_\epsilon$ be defined in (\ref{2.12}), and $\Delta L, \Delta h$ are the perturbations of $L$ and $h$, respectively. The weight matrix $W_\epsilon$ is not perturbed and therefore
 \begin{equation}
 \widehat L_\epsilon =L_\epsilon+\Delta L_\epsilon=W_\epsilon^{-1}(L+\Delta L), \qquad\widehat h_\epsilon=h_\epsilon+\Delta h_\epsilon=
W_\epsilon^{-1} (h+\Delta h),
\end{equation}
 where the norms  $\|[\Delta L\quad \Delta h]\|_F$,  $\|[\Delta L_\epsilon\quad \Delta h_\epsilon]\|_F$ of  perturbations are sufficiently small such that
 perturbed TLSE and WTLS problems have unique solutions $\widehat x_{\rm tlse}$ and  $\widehat x_\epsilon$, respectively.
In the limit sense,   $\widehat x_{\rm tlse}=\lim\limits_{\epsilon\rightarrow 0+}\widehat x_\epsilon$.

The following lemma is necessary in analyzing the first order perturbation analysis of the weighted TLS  solution $\widehat x_\epsilon$.\\

\begin{lem}\label{lem3.1}
  For the TLS problem defined in (\ref{1.0}), if the SVD of $[L\quad h]$ is given by (\ref{1.01}) and the genericity  condition (\ref{1.3}) still holds, then  we can express the first order perturbation result  in  (\ref{2.100})-(\ref{2.001}) as
$$\Delta x=K{\rm vec}([\Delta L\quad \Delta h])+{\cal O}(\|[\Delta L\quad \Delta h]\|_F^2),$$
where with  $P=L^{T}L- {\sigma}^{2}_{n+1}I_n$ and $G(x)=[x^{T}\quad -1]\otimes I_m$, the matrix $K$ has the following equivalent forms
 $$
 \begin{array}l
 K_{\rm LJ}=P^{-1}\left(2\sigma_{n+1}\rho^{-1}xu_{n+1}^TG(x)-L^TG(x)-\rho\sigma_{n+1}[I_n\quad 0_{n\times 1}]\otimes u_{n+1}^{T}\right),\\
 K_{\rm BG}=\big[-(x^{T}\otimes D)-\rho\sigma_{n+1}P^{-1}(I_n\otimes u_{n+1}^T )\qquad D\big].
 \end{array}
 $$
 Here  $ D=P^{-1}(L^T-{2\sigma_{n+1}\over \rho}xu_{n+1}^T )$ with $\rho=\sqrt{1+\|x\|_2^2}$ up to a sign $\pm 1$. The sign  is determined by the one of the $(n+1)$-th component of $v_{n+1}$, and  the value of $\rho u_{n+1}$ is unique and independent of the sign.

\end{lem}

\begin{proof} Note that in (\ref{3.1.1}), $\Big[{x\atop -1}\Big]=\rho v_{n+1}$ for $\rho=\pm\sqrt{1+\|x\|_2^2}$ and
 $$r=Lx-h=\rho[ L\quad h]v_{n+1}=\rho\sigma_{n+1}u_{n+1}, \qquad \|r\|_2^2=\rho^2\sigma_{n+1}^2,$$
 from which we observe that the value of $\rho u_{n+1}$ is uniquely determined by $r/\sigma_{n+1}$. Without loss of generality, hereafter we take $\rho=\sqrt{1+\|x\|_2^2}$.
Combining with  (\ref{1.5}), we have
 $$
\frac{2L^Trr^{T}}{\|r\|_2^2}=\frac{2\sigma_{n+1}xu_{n+1}^T}{\rho}.
 $$
 By substituting the new expression of $r$ into (\ref{3.1.1}) and (\ref{2.001}), we complete the proof.\qed
 \end{proof}

\begin{lem}\label{lem3.2}
 Let $L_\epsilon, h_\epsilon$ and $\widetilde\sigma_\epsilon$ be defined in (\ref{2.12})-(\ref{2.14}). Assume that the QR factorization of $C^T$ is given by (\ref{2.00}),
then $\lim\limits_{\epsilon\rightarrow 0+}\widetilde\sigma_\epsilon=\widetilde\sigma_{n-p+1}$ and
$$
\begin{array}l
\lim\limits_{\epsilon\rightarrow 0+}(L_\epsilon^TL_\epsilon-\widetilde\sigma_\epsilon^2I_n)^{-1}=Q_2S_{11}^{-1}Q_2^T={\cal K},\\
\lim\limits_{\epsilon\rightarrow 0+}(L_\epsilon^TL_\epsilon-\widetilde\sigma_\epsilon^2I_n)^{-1}L^TW_\epsilon^{-2}= [C_A^\dag\quad {\cal K}A^T],
\end{array}
$$
where $S_{11}=Q_2^TA^TAQ_2-\widetilde\sigma_{n-p+1}^2I_{n-p}$, $C_A^\dag, {\cal K}$ are the same as  in (\ref{2.4}).
\end{lem}

\begin{proof} The limits of $\widetilde\sigma_\epsilon$ and $(L_\epsilon^TL_\epsilon-\widetilde\sigma_\epsilon^2I_n)^{-1}$  are straightforward from the proof of Theorem 3.2 in \cite{liu}, by replacing $\mu$ there with $\epsilon^{-1}$.

For the last equality, we notice that the QR factorization of $C^T$ in (\ref{2.00}) gives
$$
\begin{array}{rl}
H_\epsilon:=\epsilon^{-2}(L_\epsilon^TL_\epsilon-\widetilde\sigma_\epsilon^2I_n)^{-1}C^T&=(C^TC+\epsilon^2A^TA-\epsilon^2\widetilde\sigma_\epsilon^2I_n)^{-1}C^T\\
&=Q[RR^T+\epsilon^2Q^T(A^TA-\widetilde\sigma_\epsilon^2I_n)Q]^{-1}R.\\
\end{array}
$$

Set $Z=Q^T(A^TA-\widetilde\sigma_\epsilon^2I_n)Q$, and partition $Z$ conformal with $RR^T$, then
$$
H_\epsilon=Q\left[\begin{array}{cc}
R_1R_1^T+\epsilon^2Z_{11}&\epsilon^2Z_{12}\\
\epsilon^2Z_{12}^T&\epsilon^2Z_{22}
\end{array}\right]^{-1}\left[\begin{array}c R_1\\0\end{array}\right]=:Q\left[\begin{array}c Y_\epsilon^{(1)}\\Y_\epsilon^{(2)}\end{array}\right],
$$
in which $Y_\epsilon^{(1)}$, $Y_\epsilon^{(2)}$ also satisfy
$$
\left[\begin{array}{cc}
R_1R_1^T+\epsilon^2Z_{11}&\epsilon^2Z_{12}\\
Z_{12}^T&Z_{22}
\end{array}\right]\left[\begin{array}c Y_\epsilon^{(1)}\\Y_\epsilon^{(2)}\end{array}\right]=\left[\begin{array}c R_1\\0\end{array}\right].
$$
Note that as $\epsilon$ tends to zero, $Z_{22}$ tends to $S_{11}$ which is nonsingular. By block Gaussian transformations to eliminate  $\epsilon^2Z_{12}$ to zero, we obtain
$$
\begin{array}l
\lim\limits_{\epsilon\rightarrow 0+}Y_\epsilon^{(1)}=\lim\limits_{\epsilon\rightarrow 0}[R_1R_1^T+\epsilon^2(Z_{11}-Z_{12}Z_{22}^{-1}Z_{12}^T)]^{-1}R_1=  R_1^{-T} ,\\
\lim\limits_{\epsilon\rightarrow 0+} Y_\epsilon^{(2)}=\lim\limits_{\epsilon\rightarrow 0}\Big(-Z_{22}^{-1}Z_{12}^TY_\epsilon^{(1)}\Big)= -S_{11}^{-1}(Q_1^TA^TAQ_2)^TR_1^{-T}.
 \end{array}
$$
Consequently,
$$
\begin{array}l
\lim\limits_{\epsilon\rightarrow 0+} H_\epsilon=Q_1R_1^{-T}-Q_2S_{11}^{-1}(Q_1^TA^TAQ_2)^TR_1^{-T}=C_A^\dag,\\
\lim\limits_{\epsilon\rightarrow 0+}(L_\epsilon^TL_\epsilon-\widetilde\sigma_\epsilon^2I_n)^{-1}L^TW_\epsilon^{-2}=\lim\limits_{\epsilon\rightarrow 0+}(L_\epsilon^TL_\epsilon-\widetilde\sigma_\epsilon^2I_n)^{-1}[\epsilon^{-2} C^T\quad A^T]\\
=\lim\limits_{\epsilon\rightarrow 0+}[H_\epsilon\quad (L_\epsilon^TL_\epsilon-\widetilde\sigma_\epsilon^2I_n)^{-1}A^T]=[C_A^\dag\quad  {\cal K}A^T],
\end{array}
$$
leading to the desired result. \qed
\end{proof}

\begin{lem}[\cite{st}] \label{lem3.3} Let an $n\times k$ matrix $X$ have full column rank $k$, and   $X$ be partitioned as $X=[X_1\quad X_2]$.
Denote $X_\epsilon=[X_1\quad \epsilon X_2]$, $Y=X_1^\dag X_2$ and $\overline X_2=X_2-X_1Y$. Then
to each singular value $s_{1}$ of $X_{1}$, there is associated a unique singular value $s_{1}^{(\epsilon)}$ of $X_{\epsilon}$ which satisfies
$s_{1}^{(\epsilon)}=s_{1}+O\left(\epsilon^{2}\right)$.
If $s_{1}$ is simple and its right singular vector is denoted by $v_{1},$ then the corresponding right singular vector of $X_{\epsilon}$ satisfies
$$
v_1^{(\epsilon)}=\left[\begin{array}{c}
{v_{1}+O\left(\epsilon^{2}\right)} \\
{\epsilon Y^{T} v_{1}+O\left(\epsilon^{3}\right)}
\end{array}\right],
$$
and the corresponding left singular vector  satisfies
$u_{1}^{(\epsilon)}=u_{1}+O(\epsilon^{2})$, where $u_{1}$ is  the left singular vector  $\overline X_2$.
Moreover, to each singular value $\bar s_2$ of $\overline X_2$, there is associated a unique singular value  $s_2^{(\epsilon)}$ of $X_\epsilon$ which satisfies
$s_2^{(\epsilon)}=\epsilon\bar s_2+O(\epsilon^3).$
If $\bar s_2$ is simple and its right singular vector is denoted by $\bar v_2$, then the corresponding right singular vector of $X_\epsilon$ satisfies
$$
v_2^{(\epsilon)}=\left[\begin{array}c -\epsilon Y\bar v_2+O(\epsilon^3)\\
\bar v_2+O(\epsilon^2)\end{array}\right],
$$
and the  corresponding left singular vector  satisfies $u_2^{(\epsilon)}=\bar u_2+O(\epsilon^2)$.
\end{lem}

{\begin{thm}\label{thm3.4}
Let  $C_A^\dag$, ${\cal K}$, $L, h$ be defined by  (\ref{2.4}) and (\ref{2.12}), respectively. Then with the notations  in (\ref{1.61})-(\ref{2.01}) and the genericity assumption (\ref{2.1}), $x=x_{\rm tlse}=\phi([L\quad h])$ is Fr$\acute{e}$chet differentiable in a neighborhood of $[L\quad h]$ and the first order estimate of $\Delta x$ is
$$
\begin{array}{rl}
\Delta x&=K_{L, h}{\rm vec}([\Delta L\quad \Delta h])+{\cal O}(\|[\Delta L\quad \Delta h]\|_F^2)\\
&=\Big(H_1G(x)-H_2\Big){\rm vec}([\Delta L\quad \Delta h])+{\mathcal O}(\|[\Delta L\quad \Delta h]\|_F^2),
\end{array}
$$
where    $G(x)=[x^T \quad -1]\otimes I_m$  for $m=p+q$ and
$$
\begin{array}l
  H_1={2\rho^{-2}{\cal K}xt^T}-[C_A^\dag\quad \quad {\cal K}A^T],\qquad
H_2={\cal K}\big([I_n\quad 0_{n\times 1}]\otimes t^T\big),
\end{array}
$$
for   $\rho=\sqrt{1+\|x\|_2^2}$,  $t^T=\big[-r^T(\widetilde A\widetilde C^\dag)  \quad r^T\big]$ with  $r=Ax-b$.
\end{thm}

\begin{proof}
We first perform the  first order perturbation analysis of the weighted TLS problem in (\ref{2.11})-(\ref{2.12}). Define the mapping
$x_\epsilon=\varphi([L_\epsilon\quad h_\epsilon])$. From Lemma \ref{lem3.1}, we have
$$
\begin{array}{rl}
\Delta x_\epsilon&=\varphi\big([L_\epsilon\quad h_\epsilon]+[\Delta L_\epsilon\quad \Delta h_\epsilon]\big)-\varphi([L_\epsilon\quad h_\epsilon])\\
&=\varphi'([L_\epsilon\quad h_\epsilon])\cdot [\Delta L_\epsilon\quad \Delta h_\epsilon]
+{\cal O}(\|[\Delta L_\epsilon\quad\Delta h_\epsilon]\|_F^2),
\end{array}
$$
 and
\begin{equation}
\Delta x_\epsilon\approx\varphi'([L_\epsilon\quad h_\epsilon])\cdot [\Delta L_\epsilon\quad \Delta h_\epsilon]=K_\epsilon {\rm vec}([\Delta L_\epsilon\quad \Delta h_\epsilon])=K_\epsilon Z_\epsilon{\rm vec}([\Delta L\quad \Delta h]),\label{3.4}
\end{equation}
where $Z_\epsilon=I_{n+1}\otimes W_\epsilon^{-1}$ and with $P_\epsilon=L_\epsilon^T L_\epsilon-\widetilde\sigma_\epsilon^2 I_n$, $G(x_\epsilon)=[x_\epsilon^T\quad  -1]\otimes I_m$,
$$
 K_\epsilon=P_\epsilon^{-1}\left(2\sigma_\epsilon\rho_\epsilon^{-1}x_\epsilon u_\epsilon^TG(x_\epsilon)-L_\epsilon^TG(x_\epsilon)-
 \rho_\epsilon\sigma_{\epsilon}[I_n\quad 0_{n\times 1}]\otimes u_{\epsilon}^{T}\right).
 $$
Here $u_\epsilon$ is the left  singular vector corresponding to the smallest nonzero singular
value $\widetilde \sigma_\epsilon$ of $\widetilde L_\epsilon:=[L_\epsilon\quad  h_\epsilon]$ and $\rho_\epsilon=(1+\|x_\epsilon\|_2^2)^{1/2}$ up to a factor $\pm 1$.

By taking the limit in (\ref{3.4}), we conclude that $x_{\rm tlse}=\phi([L\quad h])$ satisfies
$$
\phi'([L\quad h])=\lim\limits_{\epsilon\rightarrow 0+} K_\epsilon Z_\epsilon.
$$
 To prove $ \phi'([L\quad h])=K_{L, h}$, we  note that for any matrix $M_1\in {\mathbb R}^{n\times m}$, $M_2\in {\mathbb R}^{n\times m}$,
$$
 \begin{array}l
 M_1G(x_\epsilon)Z_\epsilon {\rm vec}([\Delta L\quad \Delta h])=M_1\Big([x_\epsilon^T\quad -1]\otimes W_\epsilon^{-1}\Big){\rm vec}([\Delta L\quad \Delta h])\\
 \qquad\quad\qquad\qquad \qquad \qquad=(M_1W_\epsilon^{-1})\Big([x_\epsilon^T\quad -1]\otimes I_m\Big){\rm vec}([\Delta L\quad \Delta h]),\\
 M_2\Big([I_n, 0_{n\times 1}]\otimes u_\epsilon^T\Big)Z_\epsilon{\rm vec}([\Delta L\quad \Delta h])=M_2\Big([I_n\quad 0_{n\times 1}] \otimes (u_\epsilon^TW_\epsilon^{-1})\Big){\rm vec}([\Delta L\quad \Delta h]).
  \end{array}
$$
Then
$$
\phi'([L\quad h])=\lim\limits_{\epsilon\rightarrow 0+} P_\epsilon^{-1}\Big[\Big(2\sigma_\epsilon\rho_\epsilon^{-1}x_\epsilon  u_\epsilon^TW_\epsilon^{-1}-L^TW_\epsilon^{-2}\Big)G(x_\epsilon)-\rho_\epsilon\sigma_\epsilon[I_n\quad 0_{n\times 1}] \otimes (u_\epsilon^TW_\epsilon^{-1})\Big].
$$
Here $u_\epsilon$ is the right singular vector of $[L_\epsilon\quad h_\epsilon]^T=[\epsilon^{-1}\widetilde C^T\quad \widetilde A^T]$ or $\widetilde Q^T[\widetilde C^T\quad \epsilon\widetilde A^T]$, corresponding to its smallest nonzero singular value, where $\widetilde Q$ is defined in (\ref{1.60}).   Take $X_1=\widetilde Q^T\widetilde C^T=\Big[{\widetilde R_1\atop 0}\Big]$, $X_2=\widetilde Q^T\widetilde A^T$  in Lemma \ref{lem3.3}, we obtain
$$
u_\epsilon=\left[{-\epsilon\widetilde C^{\dag^T}\widetilde A^T\bar u+{\cal O}(\epsilon^3)\atop \bar u+{\cal O}(\epsilon^2)}\right],
$$
where $\bar u$ is the left singular vector of $\widetilde A\widetilde Q_2$ corresponding to its smallest nonzero singular value and it is exactly $\widetilde u_{n-p+1}$   according to the SVD in (\ref{1.61}).

  Note that  $\widetilde u_{n-p+1}$ has a close relation to the residual vector $r$ by (\ref{1.61})-(\ref{2.2}), as revealed below
\begin{equation}
\rho\widetilde \sigma_{n-p+1}\widetilde u_{n-p+1}=\rho\widetilde A\widetilde Q_2\widetilde v_{n-p+1}=\widetilde A\left[{x\atop -1}\right]=r.\label{1.62}
\end{equation}
Combining (\ref{1.62}) with (\ref{lim}) and  Lemma \ref{lem3.2}, and for the terms in $\phi'([L\quad h])$  we obtain
$$
\begin{array}l
  \lim\limits_{\epsilon\rightarrow 0+} {2\sigma_\epsilon \rho_\epsilon^{-1}} P_\epsilon^{-1}x_\epsilon u_\epsilon^TW_\epsilon^{-1}G(x_\epsilon)\\
\qquad =2\widetilde\sigma_{n-p+1}\rho^{-1}{\cal K}x\widetilde u_{n-p+1}^T[-\widetilde A\widetilde C^\dag\quad  I_q]G(x)\\
\qquad ={2\rho^{-2}}{\cal K}xr^T[-\widetilde A\widetilde C^\dag\quad I_q]G(x)={2\rho^{-2}}{\cal K}xt^TG(x),\\
  \lim\limits_{\epsilon\rightarrow 0+} P_\epsilon^{-1}L^TW_\epsilon^{-2}G(x_\epsilon)=[C_A^\dag\quad  {\cal K}A^T]G(x),\\
  \lim\limits_{\epsilon\rightarrow 0+}\rho_\epsilon\sigma_\epsilon P_\epsilon^{-1}[I_n\quad 0_{n\times 1}] \otimes (u_\epsilon^TW_\epsilon^{-1})\\ \qquad =\rho\widetilde\sigma_{n-p+1}{\cal K}\Big([I_n\quad 0_{n\times 1}]\otimes (\widetilde u_{n-p+1}^T[-\widetilde A\widetilde C^\dag\quad  I_q])\Big)={\cal K}\big([I_n\quad 0_{n\times 1}]\otimes t^T\big).
\end{array}
$$
The assertion in the theorem then follows.\qed
 \end{proof}

By applying the similar technique on  (\ref{2.001}), we can prove another form of the perturbation result. The result is listed below, in which
$\overline K_{L,h}$   is equivalent to  $ K_{L,h}$.

\begin{thm}\label{thm3.5}
   With the notations  in Theorem \ref{thm3.4},  the first order estimate of the TLSE solution $x$ is
$$
\begin{array}{rl}
\Delta x&= \overline K_{L, h} {\rm vec}([\Delta L\quad \Delta h])+{\cal O}(\|(\Delta L\quad \Delta h)\|_F^2)\\\\
&=[(x^T\otimes H_1)- \overline H_2 \quad -H_1]\left[{\displaystyle {\rm vec}(\Delta L)\atop \displaystyle {\rm vec}(\Delta h)}\right]+{\cal O}(\|[\Delta L\quad  \Delta h]\|_F^2),
\end{array}
$$
where  $\overline H_2={\cal K}(I_n\otimes t^T)$.
\end{thm}

Let $\alpha, \beta$ be
  positive numbers, for the data space ${\mathbb R}^{m\times n}\times {\mathbb R}^{m}$,  define  the flexible norm
$$
\|[E\quad f]\|_{\cal F}=\sqrt{\alpha^2 \|E\|_F^2+\beta^2\|f\|_2}~,
$$
which is convenient to monitor the perturbations on $E$ and $f$. For instance, large values of $\alpha$ (resp. $\beta$)
enable to obtain condition number problems where mainly $f$ (resp. $E$) is perturbed. The idea of using parameter to unify the perturbations and condition numbers was first proposed in \cite{gr2}, and then used or extended by \cite{wdq,yw}.

\begin{defn}
  Let $[L\quad h]$ be defined in (\ref{2.12}), and $[\Delta L\quad \Delta h]$ is the perturbation to $[L\quad h]$. Denote $\Delta x=\phi\big([L\quad h]+[\Delta L\quad\Delta h]\big)-\phi([L\quad h])$, and define the normwise, mixed and componentwise
condition numbers as follows
$$
\begin{array}{rl}
 \kappa_{\rm n}&=\lim\limits_{\eta\rightarrow 0}\sup\Big\{{\displaystyle\|\Delta x\|_2\over\displaystyle \eta \|x\|_2}:
 \|[\Delta L \quad\Delta h]\|_{\cal  F}\le \eta \|[L\quad h]\|_{\cal   F}\Big\},\\
\kappa_{\rm m}&=\lim\limits_{\eta\rightarrow 0}\sup\Big\{{\displaystyle\|\Delta x\|_\infty\over\displaystyle \eta \|x\|_\infty}:
|[\Delta L\quad\Delta h]|\le \eta |[L\quad h]|\Big\},\\
\kappa_{\rm c}&=\lim\limits_{\eta\rightarrow 0}\sup\Big\{{\displaystyle 1\over \displaystyle \eta}\|{\displaystyle\Delta x\over\displaystyle  x}\|_\infty:
|[\Delta L\quad\Delta h]|\le \eta |[L\quad h]|\Big\},
\end{array}
$$
where $|\cdot|$ denotes the componentwise absolute value, $Y\le Z$ means $y_{ij}\le z_{ij}$ for all $i,j$, and ${Y\over Z}$ is the  entrywise division defined by ${Y\over Z}:=[{y_{ij}\over z_{ij}}]$ and  ${\xi\over 0}$
 is interpreted as zero if $\xi=0$ and infinity otherwise. The subscripts in $\kappa$ characterize the type of the condition numbers.
 \end{defn}

Write $x=\phi([L\quad h])$ as $x=\psi(g)=\bar\psi(\bar g)$ for $g={\rm vec}([L\quad h])$, $\bar g={\rm vec}([\alpha L\quad \beta h])$
and $g=(\bar D\otimes I_m)\bar g$ , in which $\bar D={\rm diag}(\alpha^{-1} I_n, \beta^{-1})$.
By following the  concept and formula for the normwise condition number in \cite{cd,ge,gk}, the condition number formulae
take the following form
$$
\begin{array}l
\kappa_{\rm n}=\frac{\displaystyle\|\bar\psi\prime(\bar g)\|_2\|\bar g\|_2}{\displaystyle\|\bar\psi(\bar g)\|_2}={\displaystyle\|K_{\alpha, \beta}\|_2\|[L\quad h]\|_{\cal  F}\over\displaystyle \|x\|_2}={\displaystyle\|\bar K_{\alpha, \beta}\|_2\|[L\quad h]\|_{\cal  F}\over\displaystyle \|x\|_2},\\
\kappa_{\rm m}=\frac{\displaystyle\||\psi\prime(g)|\cdot|g|\|_{\infty}}{\displaystyle\|\psi(g)\|_{\infty}}={\displaystyle\||K_{L,h}|\cdot {\rm vec}([|L|\quad |  h|])\|_\infty\over\displaystyle \|x\|_\infty}={\displaystyle\||\overline K_{L, h}|\cdot {\rm vec}([|L|\quad |h|])\|_\infty\over\displaystyle \|x\|_\infty},\\

\kappa_{\rm c}=\left\|{\displaystyle|\psi\prime(g)|\cdot|g|\over\displaystyle |\psi(g)|} \right\|_{\infty}=\left\| {\displaystyle|K_{L, h}|{\rm vec}([|L|\quad | h|])\over\displaystyle |x|}\right\|_\infty=\left\|{\displaystyle|\overline K_{L, h}|\cdot {\rm vec}([|L|\quad |  h|])\over\displaystyle  |x|}\right\|_\infty,
\end{array}
$$
in which  $K_{L, h}, \overline K_{L,h}$ are defined in Theorems \ref{thm3.4}  and \ref{thm3.5}, respectively, and $K_{\alpha,\beta}=K_{L, h}(\bar D\otimes I_m)$, $\bar K_{\alpha,\beta}=\overline K_{L, h}(\bar D\otimes I_m)$.

\section{Compact formula and upper bounds of condition numbers}

Note that the explicit expression of three types of condition numbers all involve the Kronecker product, which makes the storage and computation very costly.
We provide compact formula and upper bounds that are Kronecker-product free and computable.

For mixed and componentwise condition numbers, it is obvious that
$$
\begin{array}{rl}
\kappa_{\rm m}\le  {\displaystyle\|| H_1|(|  L||x|+| h|)+|{\cal K}|| L|^T|t|\|_\infty\over\displaystyle \|x\|_\infty}=:\kappa_{\rm m}^{\rm U},\\
\end{array}
$$
and
$$
\begin{array}{rl}
\kappa_{\rm c} \le  \left\|{\displaystyle| H_1|(| L||x|+| h|)+|{\cal K}||L|^T|t| \over\displaystyle  |x|}\right\|_\infty=:\kappa_{\rm c}^{\rm U}.
\end{array}
$$
 For
the normwise condition number, a compact formula for the norm $\|K_{\alpha, \beta}\|_2$ is necessary, where the matrix $K_{L,h}$  associated with $K_{\alpha,\beta}$ maps the data space ${\mathbb R}^{m\times n} \times {\mathbb R}^m$ to ${\mathbb R}^n$ in the sense that
\begin{equation}
K_{\alpha,\beta}\Delta\bar g=K_{L,h}\Delta g=H_1\Delta Lx-{\cal K}\Delta L^Tt-H_1\Delta h,\label{c1}
\end{equation}
for $\Delta\bar g={\rm vec}([\alpha\Delta L \quad \beta\Delta h])$,   $\Delta g={\rm vec}([\Delta L \quad \Delta h])$,
and the   norm
\begin{equation}
\|K_{\alpha,\beta}\|_2=\sup\limits_{\Delta \bar g\not=0}{\|K_{\alpha,\beta}\Delta\bar g\|_2\over \|\Delta \bar g\|_2}=\sup\limits_{\Delta  g\not=0}{\|K_{L,h}\Delta g\|_2\over \|[\Delta L\quad \Delta h]\|_{\cal F}}.\label{c2}
\end{equation}
The following lemma is a slight revision of Lemma 2.1 in \cite{diao}, and it gives a weighted norm estimate of an operator similar to $K_{L,h}$.\\

\begin{lem}[\cite{diao}] \label{lem4.1} Given matrices  $V\in {\mathbb R}^{m\times n}$, $X\in {\mathbb R}^{n\times m}$, $Y\in {\mathbb R}^{n\times n}$ and vectors $s\in {\mathbb R}^n, t \in {\mathbb R}^m, u\in {\mathbb R}^m$ with two positive real numbers $\alpha$ and $\beta$, for the linear operator $l$ defined by
$$
l(V,u):=-XVs+YV^Tt+Xu,
$$
its operator weighted spectral norm can be characterized by
$$
\begin{array}{rl}
\|l\|_{2,\cal F}&=\sup\limits_{V\not=0, u\not=0}{\displaystyle\|l(V,u)\|_2\over\displaystyle \|[V\quad u]\|_{\cal F}}\\
&=\left\|\Big[-{\displaystyle \|s\|_2\over\displaystyle \beta}X \quad {\displaystyle \|t\|_2\over\displaystyle \alpha}Y\Big]
\left[\begin{array}{cc} c_1I_m- c_2{\displaystyle tt^T\over\displaystyle \|t\|_2^2}&{\displaystyle\beta\over \displaystyle\alpha}{\displaystyle ts^T\over\displaystyle \|t\|_2\|s\|_2}\\0&I_n\end{array}\right]\right\|_2\\
&\le  \Big({\displaystyle \|s\|_2\over \displaystyle\beta}\|X\|_2+{\displaystyle \|t\|_2\over\displaystyle \alpha}\|Y\|_2\Big)\sqrt{\max\{1, {\displaystyle\beta^2\over\displaystyle \alpha^2}+{\displaystyle 1\over\displaystyle \|s\|_2^2}\}+{\displaystyle \beta\over \displaystyle\alpha}}.
\end{array}
$$
where $c_1= \pm\sqrt{{\displaystyle\beta^2\over\displaystyle\alpha^2}+{\displaystyle 1\over\displaystyle \|s\|_2^2}}, c_2=  c_1\pm{\displaystyle 1\over\displaystyle \|s\|_2}$.
\end{lem}

\begin{proof}  In \cite{diao}, Diao proved the value of $\|l\|_{2,{\cal F}}$ for $c_1=c_1^0, c_2=c_2^0$, where $c_1^0=\sqrt{{\beta^2\over \alpha^2}+{1\over \|s\|_2^2}}$ and $c_2^0=c_1^0+{1\over \|s\|_2}$. Let
$$
Z={\mathcal L}\Big[-{1\over \beta}\|s\|_2X \quad {1\over \alpha}\|t\|_2Y\Big],\qquad M=\left[\begin{array}{cc} c_1I_m- c_2{\displaystyle tt^T\over\displaystyle \|t\|_2^2}&\quad{\displaystyle\beta\over \displaystyle\alpha}{\displaystyle ts^T\over\displaystyle \|t\|_2\|s\|_2}\\0&I_n\end{array}\right],
$$
where $\|ZM\|_2=\|ZMM^TZ^T\|_2^{1\over 2}$ with
$$
MM^T=\left[\begin{array}{cc}
({\displaystyle \beta^2\over \displaystyle \alpha^2}+{\displaystyle 1\over \displaystyle \|s\|_2^2})I_{m}&0\\0&I_n\end{array}\right]
+{\displaystyle \beta\over \displaystyle \alpha}\left[\begin{array}{cc}
0&{\displaystyle ts^T\over\displaystyle \|t\|_2\|s\|_2}\\{\displaystyle st^T\over\displaystyle \|t\|_2\|s\|_2}&0\end{array}\right].
$$
The matrix product $MM^T$ keeps the same for   $c_1=\pm c_1^0, c_2=c_1\pm {1\over \|s\|_2}$, and $\|M\|_2^2$  has the upper bound
$$
\|M\|_2^2=\|MM^T\|_2\le  \max\{1, {\beta^2\over \alpha^2}+{1\over \|s\|_2^2}\}+{\beta\over \alpha}.
$$
The assertion of the lemma then follows.\qed
\end{proof}

\begin{thm}\label{thm4.2}    With the notations  in Theorem \ref{thm3.4},  we have the compact expression for the normwise condition number
$$
\begin{array}{l}
\kappa_{\rm n}= \left\|\Big[-{\displaystyle \|x\|_2\over\displaystyle \beta}H_1 \quad {\displaystyle \|t\|_2\over\displaystyle \alpha}{\cal K}\Big]
\left[\begin{array}{cc} c_1I_m- c_2{\displaystyle tt^T\over\displaystyle \|t\|_2^2}&{\displaystyle\beta\over \displaystyle\alpha}{\displaystyle tx^T\over\displaystyle \|t\|_2\|x\|_2}\\0&I_n\end{array}\right]\right\|_2\cdot{\displaystyle\|[L\quad h]\|_{\cal F}\over\displaystyle \|x\|_2},\label{3.22}
\end{array}
$$
where $c_1=\pm\sqrt{{\beta^2\over \alpha^2}+{1\over \|x\|_2^2}}, c_2= c_1 \pm {1\over \|x\|_2}$. The upper bound of  $\kappa_{\rm n}$ is given by
$$
\begin{array}{rl}
\kappa_{\rm n}^{\rm U}&=  \Big[{\displaystyle \|x\|_2\over \displaystyle\beta}\|H_1\|_2+{\displaystyle \|t\|_2\over\displaystyle \alpha}\|{\cal K}\|_2\Big] \times{\displaystyle\|[L\quad h]\|_{\cal F}\over\displaystyle \|x\|_2}\sqrt{\max\{1, {\displaystyle\beta^2\over\displaystyle \alpha^2}+{\displaystyle 1\over\displaystyle \|x\|_2^2}\}+{\displaystyle \beta\over \displaystyle\alpha}}\\
&\le  \Big[{\displaystyle \|x\|_2\over \displaystyle\beta}(\|C_A^\dag\|_2+\|{\cal K}A^T\|_2)+\Big({\displaystyle 2\over\displaystyle \beta}+{\displaystyle1\over\displaystyle \alpha}\Big)\|{\cal K}\|_2\|t\|_2\Big]\\
&\qquad\qquad\qquad\qquad\qquad\qquad\times{\displaystyle\|[L\quad h]\|_{\cal F}\over\displaystyle \|x\|_2}\sqrt{\max\{1, {\displaystyle\beta^2\over\displaystyle \alpha^2}+{\displaystyle 1\over\displaystyle \|x\|_2^2}\}+{\displaystyle \beta\over \displaystyle\alpha}}.
\end{array}
$$
\end{thm}

\begin{proof}  By Theorem \ref{thm3.4} and (\ref{c1}), we have
$$\Delta x=K_{L,h}\Delta g=H_1(\Delta Lx-\Delta h)-{\cal K}\Delta L^Tt+{\cal O}(\|[\Delta L\quad \Delta h]\|_F^2),
$$
for $\Delta  g={\rm vec}([ \Delta L \quad  \Delta h])$.
Ignore the high-order terms and  in Lemma \ref{lem4.1} set
$$
X=H_1, \quad V=-\Delta L,\quad Y={\cal K}, \quad s=x, \quad u=-\Delta h,\quad t^T=\big[-r^T(\widetilde A\widetilde C^\dag) \quad r^T\big],
$$
for $r=Ax-b$,  then from (\ref{c2}) and Lemma \ref{lem4.1} we get the compact expression for $\|K_{\alpha,\beta}\|_2$ as
$$
\|K_{\alpha,\beta}\|_2= \left\|\Big[-{\displaystyle \|x\|_2\over\displaystyle \beta}H_1 \quad {\displaystyle \|t\|_2\over\displaystyle \alpha}{\cal K}\Big]
\left[\begin{array}{cc} c_1I_m- c_2{\displaystyle tt^T\over\displaystyle \|t\|_2^2}&{\displaystyle\beta\over \displaystyle\alpha}{\displaystyle tx^T\over\displaystyle \|t\|_2\|x\|_2}\\0&I_n\end{array}\right]\right\|_2,
$$
and the expression for $\kappa_{\rm n}$, which is bounded by
$$
\kappa_{\rm n}\le  \Big({\displaystyle \|x\|_2\over \displaystyle\beta}\|H_1\|_2+{\displaystyle \|t\|_2\over\displaystyle \alpha}\|{\cal K}\|_2\Big){\displaystyle\|[L\quad h]\|_{\cal F}\over\displaystyle \|x\|_2}\sqrt{\max\{1, {\displaystyle\beta^2\over\displaystyle \alpha^2}+{\displaystyle 1\over\displaystyle \|x\|_2^2}\}+{\displaystyle \beta\over \displaystyle\alpha}},
$$
where $\|H_1\|_2\le 2\|t\|_2\|{\cal K}\|_2/\rho+\|C_A^\dag\|_2+\|{\cal K}A^T\|_2$. The proof is then finished.\qed
\end{proof}

\begin{remark} The upper bound   $\kappa_{\rm n}^{\rm U}$ involves the computation of $\|C_A^\dag\|_2, \|{\cal K}A^T\|_2$ and $\|t\|_2$, which can be easily implemented
from the intermediate results for solving the TLSE problem. For example, for the matrix ${\cal K}=Q_2S_{11}^{-1}Q_2^T$ with $S_{11}=(AQ_2)^T(AQ_2)-\widetilde\sigma_{n-p+1}^2I_{n-p}$,
we note that $(AQ_2)^T(AQ_2)$  is just the $(n-p)\times (n-p)$ principle submatrix of $(\widetilde A\widetilde Q_2)^T(\widetilde A\widetilde Q_2)$ and hence
$$
\begin{array}{rl}
S_{11}&=(AQ_2)^T(AQ_2)-\widetilde\sigma_{n-p+1}^2I_{n-p}=[I_{n-p}\quad 0]\Big[(\widetilde A\widetilde Q_2)^T(\widetilde A\widetilde Q_2)
-\widetilde\sigma_{n-p+1}^2I_{n-p+1}\Big]\big[{I_{n-p}\atop 0}\big]\\
&=[I_{n-p}\quad 0]\widetilde V(\widetilde \Sigma^T\widetilde \Sigma-\widetilde\sigma_{n-p+1}^2I_{n-p+1})\widetilde V^T\big[{I_{n-p}\atop 0}\big]=\widetilde V_{11}\bar S\widetilde V_{11}^T,
\end{array}
$$
where the  $(n-p)\times (n-p)$ matrix $\bar S={\rm diag}(\widetilde\sigma_i^2-\widetilde\sigma_{n-p+1}^2)$, and $\widetilde V_{11}$ is the principal $(n-p)\times (n-p)$ submatrix of $\widetilde V$.  Its inverse $S_{11}^{-1}=\widetilde V_{11}^{-T}\bar S^{-1}\widetilde V_{11}^{-1}$, where $\widetilde V_{11}$ can be cheaply computed based on  the formula in \cite[Lemma 1]{dwx}.  For the vector $t$, we can formulate $\widetilde C^\dag$ based on    the Grevill's method \cite[Chapter 7, Section 5]{big} as
$$
\widetilde C^\dag=\left[\begin{array}c \left(I_n- {\omega^{-1}{x_{\rm C}x_{\rm C}^T}}\right)C^\dag \\
{\omega^{-1} x_{\rm C}^TC^\dag }\end{array}\right],\qquad \omega= 1+\|x_{\rm C}\|_2^2,
$$
where $C^\dag$ can be easily obtained from the QR factorization of $C$.
\end{remark}

\begin{remark}  As a check, we can recover the  perturbation bound and condition numbers for the standard TLS problem, by setting $C=\Delta C=0$ and $d=\Delta d=0$ in Theorem \ref{thm3.4}.
In this case,
$$
C_A^\dag=0_{n\times p},\qquad {\cal K}=(A^TA-\sigma_{n+1}^2I_n)^{-1}=:\bar P^{-1},\qquad t^T=[0_{1\times p}  \quad r^T],
$$
where $\sigma_{n+1}=\sigma_{n+1}([A\quad   b])$, and  by setting $C=0$ and $d=0$ in (\ref{1.50})-(\ref{1.5}), we derive that $A^Tr=\sigma_{n+1}^2 x$  and $\|r\|_2^2=\sigma_{n+1}^2\rho^2$,
from which $\rho^{-2}x={ A^Tr\over \|r\|_2^2}$ and
$$
H_1=\bar P^{-1}[0_{n\times p} \quad 2\rho^{-2}xr^T-A^T]=\bar P^{-1}[0_{n\times p} \quad -A^TH_0],
$$
with $H_0=I_q-{2rr^T\over \|r\|_2^2}$ being a Householder matrix, therefore $K_{L,h}=[0_{n\times p}\quad K_{\rm tls}]$ with $K_{\rm tls}=K_{\rm LJ}$ that is defined in (\ref{3.1.1}).
The absolute normwise condition number ${\kappa}_{\rm n}^{\rm abs}=\|K_{L,h}\|_2$  reduces to the one for standard TLS problem given in \cite{jl,lj}. Moreover,
the estimate in Theorem \ref{thm3.4} becomes
$$
\begin{array}{rl}
\Delta x&=K_{\rm tls}{\rm vec}([\Delta A  \quad \Delta b])+{\cal O}(\|[\Delta A  \quad \Delta b]\|_F^2)\\
&\approx -(A^TA-\sigma_{n+1}^2I_n)^{-1}A^TH_0(\Delta Ax-\Delta b)-(A^TA-\sigma_{n+1}^2I_n)^{-1}\Delta A^Tr.
\end{array}
$$
By taking 2-norm of $\Delta x$, we obtain
the relative perturbation result of the TLS solution as follows
\begin{equation}
{\|\Delta x\|_2\over \|x\|_2}\lesssim \kappa_b{\|\Delta b\|_2\over \|b\|_2}+\kappa_A {\|\Delta A\|_2\over \|A\|_2},\label{xxw}
\end{equation}
where $\kappa_b={\|b\|_2\over \|x\|_2}\Big\|(A^TA-\sigma_{n+1}^2I)^{-1}A^T\Big\|_2$ and
$$
\kappa_A={\|A\|_2\over \|x\|_2}\Big(\|r\|_2\Big\|(A^TA-\sigma_{n+1}^2I)^{-1}\Big\|_2+\|x\|_2\Big\|(A^TA-\sigma_{n+1}^2I)^{-1}A^T\Big\|_2\Big).
$$
 The perturbation estimate in (\ref{xxw}) is the same as the result in \cite{xxw}.

For   mixed and componentwise condition numbers,  $\overline K_{L,h}$ in Theorem \ref{thm3.5} reduces to
$$
\overline K_{L,h}=x^T\otimes[0_{n\times p}  \quad  \bar D]-\bar P^{-1}(I_n\otimes [0_{1\times p}\quad r^T])-[0_{n\times p}\quad \bar D],
$$
for $ \bar D=-\bar P^{-1}(A^T-2\rho^{-2}xr^T )$, and therefore  $\kappa_{\rm m}$ and $\kappa_{\rm c}$ become
$$
\kappa_{\rm m}^{\rm tls}={\displaystyle \left\||{\cal \bar M}||A|+|\bar D||b|\right\|_\infty\over \displaystyle \|x\|_\infty},\qquad
\kappa_{\rm c}^{\rm tls}=\left\|{\displaystyle |{\cal \bar M}||A|+|\bar D||b|\over \displaystyle |x|}\right\|_\infty,
$$
where ${\cal \bar M}=(x^{T}\otimes \bar D)-\bar P^{-1}(I_n\otimes r^T )$.  These results are exactly the ones from \cite{ds}.
\end{remark}

\section{Numerical experiments}

In this section, we present numerical examples to verify   our results. The following numerical tests are performed via MATLAB
  with machine precision $u = 2.22e-16$ in a laptop with Intel Core (TM) i5-5200U CPU. In all tests, we take $\alpha=\beta=1$ in the normwise condition
  number $\kappa_{\rm n}$ and its upper bound $\kappa_{\rm n}^{\rm U}$.\\

{\bf Example 5.1} In this example,  we construct  random TLSE problems in which the entries in the input matrices  have equilibratory magnitude. Let $[A\quad b]$ be a random matrix, and the  matrix $\widetilde C=[C\quad d]$ be generated by
$$
\widetilde C=Y[{D\quad 0}]Z^T,
$$
where $Y=I_p-2yy^T$, $Z=I_{n+1}-2zz^T$, and $y\in {\mathbb R}^q, z\in {\mathbb R}^{n+1}$ are random unit vectors, $D$ is a $p\times p$ diagonal matrix whose diagonal entries are uniformly  distributed in the interval (0, 1) except the last one.  The last diagonal entry is determined such that the condition number of $\widetilde C$ is $\kappa_{\widetilde C}$.

Consider random perturbations
 \begin{equation}
  [\Delta L\quad \Delta h]=10^{-8}*{\rm rand}(p+q,n+1),\label{4.1}
  \end{equation}
and  set
$$\epsilon_{1}={\|[\Delta L\quad\Delta h]\|_F\over \|[L\quad h]\|_F},\qquad \eta_{\Delta x}^{\rm rel}={\|\Delta x-K_{L,h}{\rm vec}([\Delta L\quad \Delta h])\|_2\over \|\Delta x\|_2},$$
where $\eta_{\Delta x}^{\rm rel}$ is used to measure the correctness of Theorem \ref{thm3.4},    $\epsilon_1\kappa_{\rm n}$ with $\alpha=\beta=1$ is used to estimate an upper bound of the relative forward error ${\|\Delta x\|_2\over \|x\|_2}$.

\renewcommand\tabcolsep{25.0pt}
\begin{table}[!htb]
\begin{center}
\caption{Comparisons of forward errors via QR-SVD algorithm and upper bounds for  the perturbed TLSE problem}
\label{Table1}
\begin{tabular}{*{28}{c}}
\hline
$  \kappa_{\widetilde C}$&$\frac{\|\Delta x\|_2}{\|x\|_2}$&$\eta_{\Delta x}^{\rm rel}$&  $\epsilon_{1}\kappa_{\rm n}$ & $\epsilon_{1}\kappa_{\rm n}^{\rm U}$ &$\|C_A^\dag\|_2$\\
 \hline
   $10^2$& 7.56e-7& 1.29e-6& 3.46e-5& 2.12e-4& 5.16e+2\\
  $10^2$& 8.11e-7& 1.11e-6& 2.99e-5& 8.11e-5& 1.90e+2\\\hline

$10^3$& 2.90e-5& 5.23e-5& 5.63e-4& 9.04e-3& 2.61e+4\\
$10^3$&1.16e-5&8.10e-6&6.98e-4&7.24e-4&5.37e+3\\\hline

 $10^4$& 5.92e-5& 2.49e-4& 2.71e-3& 5.61e-3& 1.72e+4\\
 $10^4$  & 1.95e-5& 1.06e-4& 4.55e-3& 1.17e-2& 3.27e+4\\\hline

$10^5$& 1.06e-3& 2.72e-3& 5.64e-2& 1.01e-1& 4.67e+5\\
 $10^5$ & 1.67e-3& 1.86e-3& 6.31e-2& 2.04e-1& 3.34e+5\\\hline
\end{tabular}\label{4.01}
\end{center}
\end{table}

In Table \ref{Table1}, we choose $p=5, n=15$, $q=20$. For each given $ \kappa_{\widetilde C}$, we generate two different TLSE problems, and use the stable QR-SVD algorithm to compute the solutions to  the original and perturbed problems. Relative forward error of the TLSE solution is compared with the estimated upper bounds via normwise condition numbers.

The tabulated results  for $\eta_{\Delta x}^{\rm rel}$ show that the first order estimate for $\Delta x$  in Theorem \ref{thm3.4}   is reasonable. It is also observed that    ($\epsilon_1\kappa_{\rm n}$)   is about one or two orders of magnitude larger than the  actual relative forward error
${\|\Delta x\|_2\over \|x\|_2}$, even the intermediate factor $\|C_A^\dag\|_2$ is   large and the problem becomes ill-conditioned.
 The upper bound  $\kappa_{\rm n}^{\rm U}$ is a tight estimate of the normwise condition number. This indicates that $\epsilon_1\kappa_{\rm n}^{\rm U}$ can be an alternative to estimate the forward error  of the  TLSE solution.\\

{\bf Example 5.2}  In this example, we test the accuracy of the TLSE solution via  the randomized algorithm in Algorithm \ref{alg2.1}, and also compare corresponding   relative  forward error of the solution
with the estimated  upper bounds  for the perturbed TLSE problem. Let
$$
[L\quad h]=Y\left[\begin{array}c \Lambda\\ 0\end{array}\right]Z^T\in {\mathbb R}^{(p+q)\times (n+1)},\quad Y=I_{p+q}-2yy^T,\quad Z=I_{n+1}-2zz^T,
$$
where $y\in {\mathbb R}^{p+q}, z\in {\mathbb R}^{(n+1)}$ are random unit vectors, $\Lambda={\rm diag}(n,\ldots, 2, 1, \delta)$ with a  positive   parameter $\delta$ close to 0.

Set $m=50$ or $100$,  $p=0.1m, n=0.2m$, and denote
$$
E_{\rm nwtls}={\|x_{\rm nwtls}-x_{\rm qr\_svd}\|_2\over \|x_{\rm qr\_svd}\|_2 },
$$
where $x_{\rm nwtls}, x_{\rm qr\_svd}$ are   solutions of the unperturbed TLSE problem computed via NWTLS  and QR-SVD methods, respectively, and for NWTLS method in Algorithm 2.1,
we take the weighting factor  $\epsilon=10^{-8}$ and the sample size $\ell=5$.

Generate random  perturbation as in (\ref{4.1}),  and compute the solutions to the perturbed problem via NWTLS  and QR-SVD algorithms, respectively.
 In Table \ref{Table2}, we list  numerical results   with respect to different parameters.
 It is observed that the randomized NWTLS algorithm can be as accurate as the QR-SVD method. One exception for $E_{\rm nwtls}$  is for $\delta=10^{-2}$, in which case the error $E_{\rm nwtls}$
 is not close enough to machine precision. That is because when $\delta=10^{-2}$,  the ratio
  between the subdominant and dominant  eigenvalues  of $(G^TG)^{-1}$ is not small enough to guarantee   solutions with higher accuracy.
 We also note that  the upper bounds $\epsilon_1\kappa_{\rm n}$, $\epsilon_1\kappa_{\rm n}^{\rm U}$ estimated via normwise condition number are sharp  to evaluate the forward errors by the NWTLS algorithm.

\renewcommand\tabcolsep{12pt}
\begin{table}[!htb]
\begin{center}
\caption{Comparisons of forward errors via NWTLS algorithm and upper bounds for  the perturbed TLSE problem}
\label{Table2}
\begin{tabular}{*{28}{c}}
\hline
$  \delta$&$m$&{\rm cond}$(\widetilde A\widetilde Q_2)$&$E_{\rm nwtls}$&${\|\Delta x_{\rm qr-svd}\|_2\over \|x_{\rm qr-svd}\|_2}$ &${\|\Delta x_{\rm nwtls}\|_2\over \|x_{\rm nwtls}\|_2}$& $\epsilon_{1}\kappa_{\rm n}$ & $\epsilon_{1}\kappa_{\rm n}^{\rm U}$ \\
 \hline
1e-2& 50& 1.01e+3& 4.21e-11& 2.83e-8& 1.42e-8& 4.54e-7& 6.04e-7\\
1e-2 & 100& 2.03e+3& 7.84e-11& 3.34e-8& 1.67e-8& 1.27e-6& 1.59e-6\\\hline
1e-3& 50& 1.00e+4& 4.90e-15& 2.19e-8& 1.10e-8& 5.41e-7& 6.91e-7\\
1e-3  & 100& 2.03e+4& 1.25e-14& 3.15e-8& 1.58e-8& 1.05e-6& 1.22e-6\\\hline
 1e-4& 50& 9.93e+4& 1.45e-15& 3.21e-8& 1.60e-8& 4.26e-7& 5.33e-7\\
1e-4& 100& 1.98e+5& 3.09e-15& 5.11e-8& 2.56e-8& 1.35e-6& 1.71e-6\\\hline
\end{tabular}
\end{center}
\end{table}

{\bf Example 5.3} In this example, we do some numerical experiments for the TLSE problem, based on the piecewise-polynomial data fitting problem in \cite[Chapter 16]{bv} and \cite{diao2}.

 Given $N$ points $(t_i, y_i)$ on the plane, we are seeking to find a piecewise-polynomial function $f(t)$ fitting the above set of the points, where
 $$
 f(t)=\left\{\begin{array}{ll}
 f_1(t),&t\le a,\\
 f_2(t),&t>a,
 \end{array}\right.
 $$
 with $a$ given, and $f_1(t)$ and $f_2(t)$ polynomials of degree three or less,
 $$
 f_1(t)=x_1+x_2t+x_3t^2+x_4t^3,\qquad  f_2(t)=x_5+x_6t+x_7t^2+x_8t^3,.
 $$
 The conditions that $f_1(a)=f_2(a)$ and $f_1'(a)=f_2'(a)$ are imposed, so that $f(t)$ is continuous and has  a continuous first derivative at $t=a$. Suppose that  the $N$ data are numbered so that $t_1,\ldots,t_M\le a$ and $t_{M+1},\ldots, t_N>a$. The conditions $f_1(a)-f_2(a)=0$ and $f_1'(a)-f_2'(a)=0$ leads to the equality constraint $Cx=d$ for $x=[x_1\quad x_2\quad \ldots\quad x_8]^T$ and
$$
C=\left[\begin{array}{cccccccc}
1&a&a^2&a^3&-1&-a&-a^2&-a^3\\
0&1&2a&3a^2&0&-1&-2a&-3a^2\end{array}\right],\quad d=\left[\begin{array}c 0\\0\end{array}\right].
$$
The vector $x$ that minimizes the  sum of squares of the prediction errors
 $$
 \sum\limits_{i=1}^M(f_1(t_i)-y_i)^2+\sum\limits_{i=M+1}^N(f_2(t_i)-y_i)^2,
 $$
 gives $\min_x\|Ax- b\|_2$, where
 $$
 A=\left[\begin{array}{cccccccc}
 1&t_1&t_1^2&t_1^3&0&0&0&0\\
 1&t_2&t_2^2&t_2^3&0&0&0&0\\
 \vdots&\vdots&\vdots&\vdots& \vdots&\vdots&\vdots&\vdots\\
 1&t_M&t_M^2&t_M^3&0&0&0&0\\
 0&0&0&0&1&t_{M+1}&t_{M+1}^2&t_{M+1}^3\\
 0&0&0&0&1&t_{M+2}&t_{M+2}^2&t_{M+2}^3\\
 \vdots&\vdots&\vdots&\vdots& \vdots&\vdots&\vdots&\vdots\\
  0&0&0&0&1&t_{N}&t_{N}^2&t_{N}^3
  \end{array}\right],\qquad b=\left[\begin{array}c y_1\\y_2\\\vdots\\y_M\\y_{M+1}\\\vdots\\ y_N\end{array}\right],
 $$
 and the matrix $A$ is of 50\% sparsity.
Take $M=200, N=400$,  and let samples $t_i\in [0, 1]$ randomly generated such that
$$
[t_1 ~\cdots~ t_M]=a*E_{1,M},\qquad
[t_{M+1}~\cdots~ t_N]=a*1_{N-M}^T+(1-a)*E_{1,N-M},
$$
where  $E_{s,t}$ is a random $s\times t$ matrix whose entries are uniformly distributed on the interval (0,1), and
$1_{N-M}$ is a column vector of all ones.
 For a randomly generated piecewise-polynomial function $f (t)$ with a predetermined $a$, we
compute the corresponding function value $y_i=f(t_i)$, and add  random componentwise perturbations on the data
as
$$
\Delta C=10^{-8}\cdot E_{2,8}\odot C, \qquad \Delta A=10^{-8}\cdot E_{N,8}\odot A,\qquad \Delta b=10^{-8}\cdot  E_{N,1}\odot b,
$$
and $\Delta d=0$, where $\odot$ denotes the entrywise multiplication. Set
$$
\begin{array}l
\epsilon_2=\min \{\epsilon: |\Delta L|\le \epsilon |L|,  |\Delta h|\le \epsilon |h|\},
\end{array}
$$
then the relative errors ${\|\Delta x\|_\infty\over \|x\|_\infty}$, $\|{\Delta x \over x}\|_\infty$   can be bounded by $\epsilon_2 \kappa_{\rm m}, \epsilon_2\kappa_{\rm c}$ respectively.

In Table \ref{Table3}, we list the actual forward errors and corresponding upper bounds estimated via different condition numbers. It's observed that
three types condition numbers multiplied by backward errors give good estimates of the forward error when $a\ge 0.5$, and they are about one or two orders of magnitude larger than the relative forward error, and the upper bounds of three types condition numbers are  tight as well.
 The relative forward errors   of the solution are almost about $10^{-7}$, except for the case $a=0.9$.
 This happens because in (\ref{2.1}),  the gap between the singular values $\overline\sigma_{n-p}= 1.65e-4$ and   $\widetilde\sigma_{n-p+1}=1.62e-4$
is very small, and the small gap not only  makes the matrix ${\cal K}$  to have large norm, but also leads to a small value $4.87e-5$ in
the  last component of $\widetilde Q_2\widetilde v_{n-p+1}$, which also makes the solution $x$, vectors $r$ and $t$ in Theorem \ref{thm3.4} to have large norms.
These large values will magnify the backward errors and hence affect the magnitude of  $\|\Delta x\|_2$.

On the other hand, we also note that when $a$ decreases, especially for $a=0.05$, the forward error estimated via normwise condition number is far away from
the true value of  the  forward error, while mixed and componentwise
condition numbers and their upper bounds can still estimate the forward error very tightly. The reason
is that when $a$ is small, $t_1,t_2,\cdots, t_N$ are of different magnitude, and  $L$  is badly scaled,  which causes different magnitude of  entries in $K_{L, h}$.
In the forward error estimated via normwise condition number, the condition number $\|K_{L,h}\|_2 $
is dominated by its high-magnitude  entries, while  for   mixed and componentwise
condition number-based estimates,     the high-magnitude entries in $|K_{L,h}|$ might be  restrained by small or zero entries in $[|L|\quad |h|]$, leading to  tight estimates for the forward error of the solution.

\renewcommand\tabcolsep{10.0pt}
\begin{table}[!htb]
\begin{center}
\caption{Comparisons of forward error and upper bounds for  the perturbed TLSE problem}
\label{Table3}
\begin{tabular}{*{28}{c}}
\hline
$a$&$\frac{\|\Delta x\|_2}{\|x\|_2}$&$\epsilon_{1}\kappa_{\rm n}$&$\epsilon_{1}\kappa_{\rm n}^{\rm U}$ &${\|\Delta x\|_\infty\over \|x\|_\infty}$ & $\epsilon_{\rm 2}\kappa_{\rm m}$ &$\epsilon_{\rm 2}\kappa_{\rm m}^{\rm U}$ &$\|{\Delta x \over x}\|_\infty$ & $\epsilon_{\rm 2}\kappa_{\rm c}$&$\epsilon_{\rm 2}\kappa_{\rm c}^{\rm U}$\\
 \hline
0.05& 1.87e-7& 5.69e-2& 5.48e-1& 1.87e-7& 5.42e-6& 6.41e-5& 1.87e-7& 9.01e-6& 1.28e-4\\
 0.1& 3.46e-7& 9.60e-3& 8.74e-2& 3.46e-7& 6.81e-6& 7.34e-5& 3.77e-7& 9.42e-6& 1.07e-4\\
 0.3& 2.04e-7& 1.28e-4& 4.84e-4& 2.05e-7& 6.80e-6& 2.00e-5& 2.22e-7& 8.09e-6& 2.30e-5\\
  0.5&4.24e-7&5.42e-5&1.07e-4&4.36e-7&9.73e-6&2.04e-5&4.42e-7&1.04e-5&2.44e-5\\
 0.7& 1.76e-7& 2.56e-4& 5.72e-4& 1.67e-7& 5.96e-5& 1.33e-4& 2.58e-7& 6.10e-5& 1.36e-4\\
 0.9&5.25e-4&2.45e-2&1.43e-1&5.25e-4&3.30e-3&7.82e-2&1.18e-3&7.83e-3&2.96e-1\\\hline
  \end{tabular}
\end{center}
\end{table}

\section{Conclusion and future work}
In this paper, by making use of a limit technique, we present the closed formula for the first order perturbation estimate of the TLSE solution, based on which     normwise, mixed and componentwise condition numbers of problem  TLSE are   derived.  Since these expressions all involve matrix Kronecker product operations, we  propose different skills to simplify the  expressions to improve the computational efficiency.  For the normwise condition number,  the alternative expression  and upper bound  in Theorem \ref{thm4.2} is more compact and
 shown to be tight for   TLSE problems with equilibratory input data.  For the computation of $\kappa_{\rm n}$, the main cost involves  formulating $H_1, {\cal K}$, computing an $n\times (n+m)$ matrix and evaluating its  2-norm  as well.

 For   mixed and  componentwise condition numbers, the Kronecker product operation in $K_{L,h}$ also increase the storage and computational cost.
 According to the numerical experiments,  the upper bounds   are very sharp and can be  suitable to   measure the conditioning of the TLSE problem, especially for sparse and badly-scaled TLSE problems.  In the computation of the upper bounds,  the main cost   is the formulation of matrices $H_1, {\cal K}$, the remaining cost involves the cheap matrix-vector multiplications and the infinity-norm evaluation of the vectors, therefore the computation and storage is more economical than the one for normwise condition number.

In the future work, we are going to investigate the perturbation result of the TLSE
problem, in which the different magnitudes of perturbations in input data
are taken into account. The proposed results  \cite{lcz}  are expected to deliver better estimates of forward errors
of the solution than the one based on normwise condition number.  The condition number of the following multidimensional TLSE problem
$$
\min_{E,F} \|[E\quad F]\|_{F}, \quad \mbox{subject\ to}\quad (A+E)X=B+F,\quad CX=D,\label{5.1}
$$
are also considered, where  $A, E\in {\mathbb R}^{q\times n}$, $B, F\in {\mathbb R}^{q\times d}$, $C\in {\mathbb R}^{p\times n},$   $D\in {\mathbb R}^{p\times d}$.

When $C, D$ are zero matrices, the   problem reduces to the multidimensional TLS problem,  whose condition numbers have been investigated in \cite{mzw,zmw}.
It is of interest to make use of the similar limit technique to investigate the condition numbers of the multidimensional TLSE problem.  However, the discussions on  solvability  conditions and  the explicit solution to the multidimensional TLSE problem haven't been seen in the literature.
  Compared to condition numbers of the (single-right-hand) TLSE problem (\ref{1.2}), the   multidimensional case is more complicated since the smallest singular value of $\widetilde A\widetilde Q_2$ might be multiple and the corresponding singular vector is not unique and lies in an invariant subspace.
   This brings difficulty in establishing the close relation between multidimensional TLSE and multidimensional WTLS problems.  Moreover those techniques in \cite{liu} can not be used.
    New tools are needed in dealing with the  subspace approximation problem.
 We will investigate these issues  in a separate paper \cite{ljw}.\\

  {\bf Acknowledgements.} The authors  would like to thank the handling editor Prof. Claude Brezinski and  two anonymous referees for  their   constructive comments and suggestions, which greatly improve the presentation
of this paper.



\end{document}